\documentclass[reqno,a4paper,12pt]{amsart} 

\usepackage{amsmath,amscd,amsfonts,amssymb}
\usepackage{mathrsfs,dsfont}

\numberwithin{equation}{section}
\numberwithin{figure}{section}

\newcommand\R{\mathbb{R}}
\newcommand\C{\mathbb{C}}

\newcommand\Z{\mathbb{Z}}

\newcommand\T{\mathbb{T}}
\newcommand\al{\alpha}

\newcommand\del{\delta}

\newcommand\lam{\lambda}
\newcommand\Lam{\Lambda}

\newcommand\sig{\sigma}
\newcommand\Om{\Omega}

\newcommand\eps{\varepsilon}

\renewcommand\S{\mathcal{S}}

\renewcommand\le{\leqslant}
\renewcommand\ge{\geqslant}

\renewcommand\geq{\geqslant}
\newcommand\sbt{\subset}

\newcommand{\ft}[1]{\widehat{#1}}

\newcommand{\supp}{\operatorname{supp}}

\renewcommand{\Re}{\operatorname{Re}}
\renewcommand{\Im}{\operatorname{Im}}

\theoremstyle{plain}
\newtheorem{thm}{Theorem}[section]
\newtheorem{lem}[thm]{Lemma}

\newtheorem{cor}[thm]{Corollary}

\newtheorem{prop}[thm]{Proposition}

\newtheorem*{claim*}{Claim}

\newcommand{\thmref}[1]{Theorem~\ref{#1}}
\newcommand{\secref}[1]{Section~\ref{#1}}

\newcommand{\propref}[1]{Proposition~\ref{#1}}

\newcommand{\corref}[1]{Corollary~\ref{#1}}

\theoremstyle{definition}
\newtheorem{definition}[thm]{Definition}
\newtheorem*{definition*}{Definition}
\newtheorem*{remarks*}{Remarks}
\newtheorem*{remark*}{Remark}

\newenvironment{enumerate-roman}
{\begin{enumerate}
\addtolength{\itemsep}{5pt}
}
{\end{enumerate}}

\newenvironment{enumerate-num}
{\begin{enumerate}
\addtolength{\itemsep}{4pt}
}
{\end{enumerate}}

\newenvironment{enumerate-alph}
{\begin{enumerate}
\addtolength{\itemsep}{4pt}
}
{\end{enumerate}}

\begin{document}

\title
[Completeness of uniformly discrete translates in $L^p(\mathbb{R})$]
{Completeness of uniformly discrete \\ translates in $L^p(\mathbb{R})$}

\author{Nir Lev}
\address{Department of Mathematics, Bar-Ilan University, Ramat-Gan 5290002, Israel}
\email{levnir@math.biu.ac.il}

\date{January 4, 2025}
\thanks{Research supported by ISF Grant No.\ 1044/21.}

\begin{abstract}
We construct a real sequence $\{\lambda_n\}_{n=1}^{\infty}$ satisfying $\lambda_n = n +  o(1)$,  and a Schwartz function $f$ on $\mathbb{R}$, such that for any $N$ the system of translates $\{f(x - \lambda_n)\}$, $n > N$, is complete in the space $L^p(\mathbb{R})$ for every $p>1$. The same system is also complete in a wider class of Banach function spaces on $\mathbb{R}$.
\end{abstract}

\maketitle


\section{Introduction}

\subsection{}
A result due to Atzmon and Olevskii \cite{AO96} asserts that
for every $p>2$, there is a function $f \in L^p(\R)$
whose translates  by the positive integers 
$\{f(x-n)\}$, $n=1,2,\dots$,
span the whole space $L^p(\R)$, that is,
these translates are complete in  $L^p(\R)$.
There are now several known approaches to this result,
see \cite{Nik99}, \cite{FOSZ14}.

In the space $L^2(\R)$, on the other hand,
integer translates can never be complete.
In this connection, it was conjectured (see \cite{RS95}) 
that no system
$\{f(x-\lam)\}$, $\lam \in \Lam$,
can be complete in $L^2(\R)$ if the set $\Lam$
is \emph{uniformly discrete}, i.e.\   satisfies the condition
\begin{equation}
\label{eq:uddef}
  |\lam' - \lam| \ge  \del(\Lam) >0
\end{equation}
for any two distinct points  $\lam, \lam' \in \Lam$.
However, this is not the case. It was proved by Olevskii
\cite{Ole97} that for any ``small perturbation'' of the integers, 
\begin{equation}
\label{eq:pertz}
\lam_n = n + \alpha_n, \quad  
0 \ne \alpha_n \to 0 \quad (|n| \to +\infty)
\end{equation}
there exists  $f \in L^2(\R)$
 such that  the system
$\{f(x-\lam_n)\}$, $n \in \Z$,
is complete in $L^2(\R)$. It was moreover
 shown in \cite{OU04} 
that if the perturbations are exponentially small, i.e.\
\begin{equation}
\label{eq:exppertn}
0 < |\alpha_n | < C r^{|n|}, \quad n \in \Z,
\end{equation}
for some $0<r<1$ and $C>0$, 
then  $f$ can be chosen in the Schwartz class.

More recently, by a development of the approach
from \cite{OU04}, the latter result was
extended  to $L^p(\R)$ spaces \cite{OU18a},  namely,
it was proved that there is a  function
 $f$ in the Schwartz class, such that if the sequence
 $\{ \lam_n \}$, $n \in \Z$, 
satisfies \eqref{eq:pertz} and \eqref{eq:exppertn} 
then the system  $\{f(x-\lam_n)\}$, $n \in \Z$,
is complete in $L^p(\R)$ for every $p>1$.

In fact the latter result holds for a wide class of
separable Banach function spaces on $\R$
\cite{OU18b}. On the other hand, in the space
 $L^1(\R)$ no  system of uniformly discrete 
translates can be complete, see \cite{BOU06}.

\subsection{}
The goal of the present note is to provide a different approach
to the construction of a function $f$ which spans the
space  $L^p(\R)$, $p>1$, by a uniformly discrete system of translates.
In fact, in our construction we shall use only positive translates, and 
moreover the completeness remains true for any subsystem 
obtained by the removal of a finite number of elements.

The result can be stated as follows:

\begin{thm}
\label{thm:A1}
There is a real sequence 
$\{\lam_n\}_{n=1}^{\infty}$ satisfying
$ \lam_n = n +  o(1)$,  and there is 
a Schwartz function $f$ on $\R$, 
such that for any $N$ the system 
\begin{equation}
\label{eq:ftranslama1}
\{f(x - \lam_n)\}, \; n>N,
 \end{equation}
is complete in the space $L^p(\R)$ for every $p>1$.
\end{thm}

Moreover, our approach allows us to extend the result to the wider 
class of Banach function spaces on $\R$ considered in \cite{OU18b},
see \secref{sec:comtranspaces}.

The approach  is inspired by the papers \cite{Lan64} and \cite{AAG08}.


\section{Preliminaries}

In this section we briefly recall some necessary background
in the theory of Schwartz distributions on $\R$
(see \cite{Rud91} for more details).

The \emph{Schwartz space}  $\S(\R)$ consists of all infinitely smooth
functions $\varphi$ on $\R$ such that for each $n,k \geq 0$, the seminorm
\begin{equation}
\|\varphi\|_{n,k} := \sup_{x \in \R} (1+|x|)^n  |\varphi^{(k)}(x)|
\end{equation}
is finite.  It is a topological linear space whose topology is
induced by the metric
\begin{equation}
\label{eq:cinfmetric}
d(\varphi, \psi) := \sum_{n,k \ge 0}
2^{-(n+k)} \frac{  \| \varphi - \psi \|_{n,k} }{1 + \| \varphi - \psi \|_{n,k}}
\end{equation}
which also makes $\S(\R)$ a complete, separable metric space.

A \emph{temperate distribution} on $\R$ is a 
linear functional  on the Schwartz space   $\S(\R)$
which is continuous with respect to the 
metric \eqref{eq:cinfmetric}.
We use  $\alpha(\varphi)$ to denote  the action of
a  temperate distribution $\alpha$ on a
Schwartz function $\varphi$.

If $\varphi$ is a Schwartz function  then its Fourier transform
is defined by
\begin{equation}
\ft{\varphi}(x) = \int_{\R} \varphi(t) e^{-2\pi i x t}  dt.
\end{equation}
If $\alpha$ is a temperate distribution then
its Fourier transform is defined by 
$\ft{\alpha}(\varphi) = \alpha(\ft{\varphi})$.

We denote by $\supp(\alpha)$ the closed support of a
temperate distribution $\alpha$.

If $\alpha$ is a temperate distribution with compact
support, then $\alpha(\varphi)$ is well defined
for every smooth function $\varphi$ on $\R$.
In this case $\ft{\al}$
is an infinitely smooth function  given by
$\ft{\al}(x) = \al(e_{-x})$, $x \in \R$,
where here we denote
$e_{x}(t) := e^{2\pi i x t}$.

If $\alpha$ is
a temperate distribution  and if
$\varphi$  is a 
Schwartz function, then 
the product $\alpha \cdot \varphi$ is 
a temperate distribution   defined
by $(\alpha \cdot \varphi)(\psi) =
\alpha(\varphi \cdot \psi)$,
$\psi \in \S(\R)$.
In this case we have 
$\supp(\alpha \cdot \varphi) \sbt
\supp(\alpha) \cap \supp(\varphi)$.

The convolution
$\alpha \ast \varphi$ 
of a temperate distribution  
 $\alpha$ and 
a Schwartz function $\varphi$ is 
 an infinitely smooth function
which is also
a temperate distribution, and
whose Fourier transform is $\ft{\al} \cdot \ft{\varphi}$.
If $\varphi$ has compact support, then
$\supp(\al \ast \varphi)$ is contained in the
 Minkowski sum 
 $\supp(\al) + \supp(\varphi)$.


\section{Landau's complete system of exponentials}

The starting point for our approach
is H.\ Landau's paper \cite{Lan64}.
The paper  is concerned with
the completeness of a system of exponential functions
in the space
$L^p(\Om)$, $1 \le p < \infty$, or $C(\Om)$,
where $\Om$ is an interval, or a finite union of intervals,
in $ \R$.

By classical theory it is known that if $\{\lam_n\}$, $n \in \Z$, 
is a ``regular'' sequence, i.e.\ 
\begin{equation}
\label{eq:landistint}
\sup_{n \in \Z} |\lam_n - n| < +\infty,
\end{equation}
then the system of exponential functions 
$\{e^{2 \pi i \lam_n t}\}$, $n \in \Z$,
is complete on any interval   of length $<1$, 
while it  is incomplete on any interval  of length $>1$
(see e.g. \cite[Section 4.7]{OU16}). These results are
based on the theory of entire functions of exponential type.

It was discovered by Landau that 
 the situation is quite different if the single interval is 
 replaced by a finite union of intervals. 
 The following result was proved in \cite{Lan64}.

\begin{thm}[{\cite{Lan64}}]
\label{thm:landaucmpthm}
Given   $\eps > 0$ there is  a real sequence $\{\lam_n\}_{n=1}^{\infty}$ with
$|\lam_n   - n| < \eps$, such that 
for any finite union of intervals of the form
\begin{equation}
\label{eq:omland}
\Om_{h,K} :=  \bigcup_{|k| \le K} [k + h, k + 1 - h],
\qquad \text{($K$ finite, $0<h<1/2$)}
\end{equation}
and for any $N$, the system
$\{ e^{2 \pi i \lam_n t} \}$, $n > N$,
is complete in the space $C(\Om_{h,K})$.
\end{thm}

We note that a  set of the form \eqref{eq:omland} can have arbitrarily large measure.

The sequence $\{\lam_n\}_{n=1}^{\infty}$ in \thmref{thm:landaucmpthm}
is constructed as follows. First, the positive integers are partitioned
into pairwise disjoint subsets $\{S_r\}_{r=1}^{\infty}$
such that
\begin{equation}
\limsup_{k \to +\infty} \frac{ \# (S_r \cap [1, k]) }{k} = 1,
\quad r=1,2,\dots.
\end{equation}
Now, let $\{\theta_r\}_{r=1}^{\infty}$ be real numbers
satisfying $|\theta_r| < \eps$, and for each positive integer $n$, define
$\lam_n := n + \theta_r$ if  $n \in S_r$. 
It is shown in \cite{Lan64} that there is 
a choice of  the numbers $\{\theta_r\}$, such that the sequence
$\{\lam_n\}_{n=1}^{\infty}$ 
 thus obtained satisfies the conclusion in
\thmref{thm:landaucmpthm}.
(In fact, it suffices that the numbers $\{\theta_r\}$
be distinct modulo $1$).

By duality, Landau's result can be equivalently stated as a uniqueness 
result for (complex) measures $\mu$ whose  support lies in a set
 of the form \eqref{eq:omland} and whose Fourier transform 
 $\ft{\mu}$ vanishes on the set $\{- \lam_n\}$, $n > N$.
  We will need the following
 slightly more general version of the result,
 where ``measure'' is replaced by ``distribution''. 

\begin{cor}
\label{cor:landauuniqset}
Given   $\eps > 0$, let
$\{\lam_n\}_{n=1}^{\infty}$ 
be the sequence from \thmref{thm:landaucmpthm}.
Let $\alpha$ be a temperate
 distribution on $\R$ whose support 
 lies in a set $\Om_{h,K}$ of the form \eqref{eq:omland},
and assume that there is $N$ such that
$  \ft{\alpha}(- \lam_n) = 0$, $n>N$.
Then $\alpha  = 0$.
\end{cor}

\begin{proof}
Let $\chi$ be a smooth nonnegative function,
 $\supp(\chi) \sbt [-1,1]$,
 $\int \chi = 1$, and define $\chi_\del(t) := \del^{-1} \chi(\del^{-1}  t)$.
The convolution $\alpha_\del :=  \alpha \ast \chi_\del$ is 
a measure (in fact, a  smooth function) 
supported on the set $ \Om_{h-  \del, K}$
and annihilating the system
$\{ e^{2 \pi i \lam_n t} \}$, $n > N$.
By the completeness of  the system in the
space  $C(\Om_{ h-  \del, K})$  it follows that
$\alpha_\del = 0$ for $0<\del<h$.
 Letting $\del \to 0$, we obtain $\al = 0$
(see \cite[Theorem 6.32]{Rud91}).
\end{proof}


\section{Completeness of weighted exponentials}
\label{sec:comreal}

\subsection{}
The next step in our approach is 
to construct a system of 
\emph{weighted exponential functions} 
with ``almost integer'' frequencies, 
which is complete in  a certain
 space of smooth functions on $\R$ 
that will be now introduced.

\begin{definition}
The space $I_0(\R)$ is defined to be 
the closed   subspace 
 of  the Schwartz space $\S(\R)$, that consists
of the functions $f \in \S(\R)$  satisfying 
\begin{equation}
\label{eq:iocrit}
f^{(j)}(n) = 0, \quad n \in \Z, \quad j=0,1,2,\dots.
\end{equation}
\end{definition}

The space $I_0(\R)$ is then a 
complete, separable metric space
with the metric inherited from the Schwartz space.

We will use Landau's \thmref{thm:landaucmpthm} in order
to prove the following result:

\begin{thm}
\label{thm:landaucomp2}
There is  a real sequence $\{\lam_n\}_{n=1}^{\infty}$ with
$\lam_n  = n + o(1)$, and there is a function $\varphi \in I_0(\R)$,
such that for any $N$ the system 
\begin{equation}
\label{eq:phiexpsyslam2}
\{\varphi(t) e^{2 \pi i \lam_n t}\}, \; n>N,
 \end{equation}
is complete in the space $I_0(\R)$.
\end{thm}

The proof will be done in several steps.

\subsection{}
We use $J_0(\R)$ to denote the linear  space consisting of
all the smooth, compactly supported functions  on $\R$
which  vanish in a neighborhood of $\Z$.

\begin{lem}
\label{lem:j0dense}
$J_0(\R)$ is a dense subspace of $I_0(\R)$.
\end{lem}

\begin{proof}
We choose a smooth, compactly supported
 function $\chi$ on $\R$ satisfying
$\chi(t)=1$ on $[-1,1]$, and define
$\chi_r(t) := \chi(rt)$. First, 
assume that a smooth, compactly supported
function $f$ has a zero of infinite order
at the origin. It is then straightforward to check 
that $\|f \cdot \chi_r\|_{n,k} \to 0$
as $r \to +\infty$ for every $n,k$,
hence $f \cdot (1 - \chi_r) \to f$
as $r \to +\infty$ in the Schwartz space metric.
This shows that
 $f$ is a  limit of smooth, compactly supported
  functions  vanishing in a neighborhood of the origin.
By a similar argument, any compactly supported
function $f \in I_0(\R)$ is a limit 
of functions  belonging to $J_0(\R)$.
So to conclude the proof, it would suffice
to show that the 
compactly supported functions in $I_0(\R)$
form a dense set in $I_0(\R)$. Indeed,
if $f \in I_0(\R)$ then 
$f \cdot \chi_r$ is  
compactly supported, belongs to $I_0(\R)$
and $f \cdot \chi_r \to f$ as $r \to 0$
(see \cite[Theorem 7.10]{Rud91}).
\end{proof}

\subsection{}
We will next prove the following 
version of \thmref{thm:landaucomp2}.

\begin{thm}
\label{thm:landaucomp1}
Given   $\eps > 0$, let
$\{\lam_n\}_{n=1}^{\infty}$ 
be the sequence from \thmref{thm:landaucmpthm}.
Then there is a function $\varphi \in I_0(\R)$,
such that for any $N$ the system 
\eqref{eq:phiexpsyslam2} 
is complete in the space $I_0(\R)$.
Moreover the set of such functions $\varphi$
is residual  in $I_0(\R)$.
\end{thm}

We note that the sequence
$\{\lam_n\}_{n=1}^{\infty}$ in this result
is uniformly discrete, but does not satisfy the condition 
$ \lam_n = n +  o(1)$, so \thmref{thm:landaucomp2}
does not follow directly.

The proof  below is based on the method from \cite[Section 2]{AAG08}.

\begin{proof}[Proof of \thmref{thm:landaucomp1}]
For each   $\chi \in I_0(\R)$ and each $N$ and $\rho>0$,
let $G(\chi, \rho, N)$ be the set of all functions
$\varphi \in I_0(\R)$ such that there is a polynomial
\begin{equation}
\label{eq:polyplamn}
P(t) = \sum_{n > N} c_n e^{2 \pi i \lam_n t}
\end{equation}
satisfying $d(\varphi \cdot P, \chi) < \rho$.
We claim that $G(\chi, \rho, N)$ is an open, dense 
set in $I_0(\R)$.

To see that $G(\chi,\rho,N)$ is an open set, 
we observe that 
multiplication by a polynomial $P$ is a continuous mapping 
$I_0(\T) \to I_0(\T)$, hence
$\{ \varphi  \in I_0(\T) :  d(\varphi \cdot P, \chi) < \rho\}$
is an open set for every polynomial $P$.
It follows that $G(\chi, \rho, N)$, being  a union of sets of this form
as $P$ goes through all polynomials \eqref{eq:polyplamn}, 
is an open set.

Next we show that the set
$G(\chi,\rho,N)$ is   dense 
 in $I_0(\R)$. To establish this we
fix $f \in I_0(\R)$ and $\delta>0$, and show
that there is $\varphi \in G(\chi,\rho,N)$ such that
$d(f,  \varphi) < \delta$.

Since $J_0(\R)$ is a dense subspace of $I_0(\R)$, we can find
two  smooth functions 
$\chi_0$ and $f_0$  on $\R$, 
 both with compact support and
 vanishing in a neighborhood of $\Z$, such that
$d(\chi, \chi_0 ) < \rho/2$ and 
$d(f, f_0) < \del/2$. 
We  choose $R$ sufficiently large and
 $h>0$ sufficiently small, such that
 $\chi_0$ and $f_0$ are both supported on $[-R,R]$
and both vanish on the set $\Z + (-2h, 2h)$.
We also choose 
a smooth, compactly supported function $\eta$ on $\R$
vanishing on the set $\Z + (-h, h)$ and
such that $\eta(t) = 1$ on $[-R,R] \setminus
(\Z + (-2h, 2h))$.

We claim that there is a Schwartz function $q$ 
such that
$d(f_0, q \cdot \eta)  < \del / 2$
and moreover, $q$ has no zeros on $[-R,R]$.
Indeed, the fact that $\eta(t) = 1$
on $\supp(f_0)$ implies that
$f_0  =  f_0 \cdot \eta$, hence
by approximating $f_0$ in the Schwartz space
 by a nonzero smooth function  $g$ such that
$\ft{g}$ has compact support, we obtain
 $d(f_0, g \cdot \eta)  < \del / 2$.
Observe that the function $g$ admits a holomorphic
extension to the complex plane, given by
\begin{equation}
g(z) := \int \ft{g}(u) e^{2 \pi i u z} du, \quad z \in \C,
\end{equation}
and $g$ being nonzero has only finitely many zeros
in the closed rectangle given by $|\Re(z)| \le R$ and 
$|\Im(z)| \le 1$. We then
 define $g_\sig(t) := g(t + i \sig)$
 and observe that
 $g_\sig \to g$ as $\sig \to 0$ in the Schwartz space,
so the function $q  := g_\sig$
satisfies the required properties
provided that $\sig>0$ is small enough.

Let $\varphi := q \cdot \eta $, then
$\varphi \in I_0(\R)$ (in fact, $\varphi \in J_0(\R)$)  and
\begin{equation}
d(f,  \varphi) \le d(f,   f_0) +  d(f_0,   q \cdot \eta)
< \del/2 +  \del/2 = \del.
\end{equation}

We claim that
$\varphi \in G(\chi,\rho,N)$. Indeed,
 suppose that $\alpha$  is a  temperate distribution on $\R$ which
 annihilates   the system \eqref{eq:phiexpsyslam2}.
  This means that  the distribution
$\alpha \cdot \varphi$ satisfies 
 $(\alpha \cdot \varphi)^{\wedge}(- \lam_n) = 0$ 
for all $n > N$. 
Since $\varphi$ is compactly supported  and vanishes 
on $\Z + (-h, h)$, the distribution
$\alpha \cdot \varphi$ thus satisfies the assumptions 
in \corref{cor:landauuniqset}.  It therefore
follows from \corref{cor:landauuniqset}  that 
$\alpha \cdot \varphi =0$. Since
$\varphi$ has no zeros on $\supp(\chi_0)$, we can write
$\chi_0  =  \varphi \cdot g$ where $g$  is a smooth
 function of compact support. Hence
\begin{equation}
\alpha(\chi_0) = \alpha(\varphi \cdot g ) = 
(\alpha \cdot \varphi)(g) = 0.
\end{equation}
We conclude (see  \cite[Theorem 3.5]{Rud91}) that the function
$\chi_0$ must lie in the closed linear subspace of $I_0(\R)$
spanned by the system \eqref{eq:phiexpsyslam2},
so there is a polynomial
\eqref{eq:polyplamn}
satisfying $d(\varphi \cdot P, \chi_0) < \rho/2$.
In turn this implies that
\begin{equation}
d(\varphi \cdot P, \chi) \le
d(\varphi \cdot P, \chi_0) +
d(\chi, \chi_0) < \rho/2  + \rho/2 = \rho,
\end{equation} 
so that $\varphi \in G(\chi,\rho,N)$.
This establishes that 
$G(\chi,\rho,N)$ is a dense set in $I_0(\R)$.

Finally,   we choose a sequence $\{\chi_{j}\}$
dense in $I_0(\R)$, and consider the set
\begin{equation}
\label{eq:resid}
\bigcap_{j,k,N}^{\infty} G(\chi_j, k^{-1}, N).
\end{equation} 
It is a residual set in $I_0(\R)$, being
the countable intersection of open, dense sets,
and any function  $\varphi$ 
from the set \eqref{eq:resid} 
satisfies the conclusion of \thmref{thm:landaucomp1}.
\end{proof}

\subsection{}
\label{subsec:ss1}
Now we use \thmref{thm:landaucomp1} in order to
establish \thmref{thm:landaucomp2}.

\begin{proof}[Proof of \thmref{thm:landaucomp2}]
Let $\{\chi_k\}_{k=1}^{\infty}$ be a sequence 
which  is  dense in the space $I_0(\R)$.
 We then construct by induction functions $\varphi_k \in I_0(\R)$,
 an increasing sequence of positive integers $\{N_k\}$,
 and trigonometric polynomials 
\begin{equation}
\label{eq:pknkdeflmnk}
P_k(t) = \sum_{N_k < n < N_{k+1}} c_n  e^{2 \pi i \lam^{(k)}_n t} 
\end{equation} 
in the following way. We let $\varphi_0 := 0$ and $N_0 := 0$.

At the $k$'th step of the induction, we invoke
 \thmref{thm:landaucmpthm}  with $\eps = k^{-1}$
and obtain a real sequence
$\{\lam^{(k)}_n\}_{n=1}^{\infty}$ 
with
$|\lam^{(k)}_n   - n| < k^{-1}$, for all $n$.
By  \thmref{thm:landaucomp1}, there exists 
in the space $I_0(\R)$ a dense (in fact, residual)
set of functions $\varphi$,
such that for any $N$ the system 
$\{\varphi(t) e^{2 \pi i \lam^{(k)}_n t}\}$, $n>N$,
is complete in  $I_0(\R)$. This implies that we can 
choose $\varphi_k \in I_0(\R)$ and
a polynomial $P_k$ as in \eqref{eq:pknkdeflmnk}, such that
$d(\varphi_k \cdot P_k, \chi_{k}) < k^{-1}$
 and 
\begin{equation}
\label{eq:phiphikplusB1}
d(\varphi_k, \varphi_{k-1}) \le  2^{-k}, \quad
\max_{1 \le l \le k-1} d(\varphi_k  \cdot P_l, \varphi_{k-1} \cdot P_l ) \le  k^{-1} 2^{-k}.
\end{equation} 
The sequence 
$\{ \varphi_k  \}$ then
converges in the space $I_0(\R)$ to some   
$ \varphi \in I_0(\R)$, and 
\begin{equation}
\label{eq:phkchikapproxB2}
d(\varphi \cdot P_k,  \chi_{k})
\le d(\varphi_k \cdot P_k,  \chi_k)
+   \sum_{l=k+1}^{\infty} d(\varphi_l \cdot P_k,  \varphi_{l-1} \cdot P_k)  < 2 k^{-1}
\end{equation}
for every $k$, due to \eqref{eq:phiphikplusB1}.

The intervals $(N_{k}, N_{k+1})$, $k=1,2,\dots$, are pairwise
disjoint, so each positive integer $n$ may belong to at most
one  of these intervals. For each $n \in
(N_{k}, N_{k+1}) \cap \Z$
we define $\lam_n := \lam^{(k)}_n$, and if $n$ does not
lie in any of the intervals $(N_{k}, N_{k+1})$,
then we set $\lam_n := n$. It follows that
$|\lam_n - n| < k^{-1}$ for $n > N_k$, 
which implies that $\lam_n = n + o(1)$.

Now given any $N$, we  choose a sufficiently large
 $k_0$ such that $N_{k_0} > N$. Then the sequence 
 $\{\chi_{k}\}$, $k > k_0$, is also dense in $I_0(\R)$.
Since we have
$d(\varphi \cdot P_k,  \chi_{k}) < 2 k^{-1}$
according to \eqref{eq:phkchikapproxB2},
 we conclude that the sequence 
$\{ \varphi \cdot P_k \}$,  $k > k_0$, 
is  dense in $I_0(\R)$ as well. But notice that 
$ \varphi \cdot P_k$ belongs to the linear
span of the system \eqref{eq:phiexpsyslam2}
if $k > k_0$,  due to 
 \eqref{eq:pknkdeflmnk}.
 This implies that the system \eqref{eq:phiexpsyslam2}
is complete in $I_0(\R)$.
\end{proof}


\section{Completeness of translates}
\label{sec:comtranspaces}

Now we can prove
\thmref{thm:A1}. Moreover,
 we will prove an extension of this theorem to the wider 
class of Banach function spaces on $\R$ considered in \cite{OU18b}.

\subsection{}
Let $X$ be a Banach function space on $\R$.
Following \cite{OU18b} we assume that
\begin{enumerate-roman}
\item
\label{it:ou1}
 The Schwartz space $\S(\R)$
is continuously and densely embedded in $X$.
\end{enumerate-roman}
In this case every element $\alpha$
of the dual space $X^*$ defines a continuous 
linear functional on $\S(\R)$, so $\alpha$
 is a temperate distribution on $\R$.
 We suppose further that
\begin{enumerate-roman}
\setcounter{enumi}{1}
\item
\label{it:ou2}
If   a temperate distribution $\alpha \in X^*$  satisfies
 $\supp(\ft{\alpha}) \sbt \Z$,
then $\alpha = 0$.
\end{enumerate-roman}
We can now establish the following result
(compare with \cite[Theorem 1]{OU18b}).

\begin{thm}
\label{thm:B1}
There is a real sequence 
$\{\lam_n\}_{n=1}^{\infty}$ satisfying
$ \lam_n = n +  o(1)$,  and there is
a Schwartz function $f$ on $\R$, 
such that for any $N$ the system 
\begin{equation}
\label{eq:ftranslamb1}
\{f(x - \lam_n)\}, \; n>N,
 \end{equation}
is complete in every Banach function space $X$ 
satisfying \ref{it:ou1} and \ref{it:ou2} above.
\end{thm}

\begin{proof}
Consider the sequence $\{\lam_n\}_{n=1}^{\infty}$
and the function $\varphi \in I_0(\R)$ from
\thmref{thm:landaucomp2}.
Let $f := \ft{\varphi}$ which
  is a Schwartz  function on $\R$. 
  Then the Fourier transform maps  the 
weighted exponential system 
$\{\varphi(t) e^{2 \pi i \lam_n t}\}$
onto the system of translates
 $\{f(x - \lam_n)\}$.
  
Let now $X$ be a Banach function space 
satisfying \ref{it:ou1} and \ref{it:ou2}.
The Fourier transform is an isomorphism of
the Schwartz space  (as a  topological linear space)
onto itself, so it
 follows from  \ref{it:ou1} that the Fourier transform 
  embeds  $I_0(\R)$    continuously into $X$.
  Hence, to prove that the system \eqref{eq:ftranslamb1}
  is complete in $X$, 
  it would suffice to show that the image 
$\ft{I}_0(\R)$ of the space  $I_0(\R)$ 
  under the Fourier transform, is a dense subspace of  $X$.
  
  Indeed, let   $\alpha \in X^*$  be
   a temperate distribution   which annihilates   
  $\ft{I}_0(\R)$, that is, 
  $\alpha(\ft{\psi}) = 0$ for every $\psi \in I_0(\R)$.
This means that the Fourier transform 
  $\ft{\alpha}$ annihilates the space $I_0(\R)$,
  which is possible only if $\supp(\ft{\al}) \sbt \Z$. 
  It then follows from \ref{it:ou2} that $\alpha =0$.
  So this establishes  that 
    $\ft{I}_0(\R)$  is a dense subspace of  $X$,
    and proves the theorem. 
\end{proof}

\subsection{}
Finally,  
\thmref{thm:A1} follows as a special case of
\thmref{thm:B1}. Indeed, we have

\begin{prop}
\label{prop:xlp}
The space $X  = L^p(\R)$, $p>1$, 
satisfies conditions \ref{it:ou1} and \ref{it:ou2} above.
\end{prop}

This is proved in \cite[Section 3]{OU18b} as a consequence
of a more general result. 
Below we give a short, self-contained proof of this fact.

\begin{proof}[Proof of \propref{prop:xlp}]
Condition \ref{it:ou1} follows from the inequality 
$\|f\|_{L^p(\R)} \le C_p \|f\|_{1,0}$
which holds for every Schwartz function $f$,
where $C_p = (\int_{\R} (1+|x|)^{-p} dx)^{1/p}$ 
is a finite constant that depends
on $p$ but does  not depend on $f$.

To establish   condition \ref{it:ou2}, we recall that
a temperate distribution $\alpha \in (L^p(\R))^*$
is an element of the space  $L^q(\R)$, $q = p/(p-1)$.
Assume that  $\supp(\ft{\alpha}) \sbt \Z$,
 and let $\chi$ be a Schwartz function on $\R$ with
 $\supp(\ft{\chi}) \sbt [- \frac2{3}, \frac2{3}]$ and
 $\ft{\chi}(t)=1$ on $ [- \frac{1}{3}, \frac{1}{3}]$.
 The distribution $\ft{\al} \cdot \ft{\chi}$ is then
 supported at the origin  and
 coincides with $\ft{\al}$ in a neighborhood of the origin.
  It is well known that a distribution supported at
  the origin is a finite linear combination of derivatives
  of Dirac's measure at the origin
  (see \cite[Theorem 6.25]{Rud91}). In turn this
  implies that $\al \ast \chi$ must be a polynomial.
  But $\al \ast \chi \in L^q(\R)$, so this is possible only if 
$\al \ast \chi = 0$. We conclude that $\ft{\al}$
vanishes in a neighborhood of the origin.
By a similar argument, $\ft{\al}$ must
vanish also  in a neighborhood of any integer other
than the origin. 
Hence $\ft{\al}$, and thus also $\al$, is zero,
which establishes the condition \ref{it:ou2}.
\end{proof}

 We refer the reader to  \cite[Section 3]{OU18b}
 where other examples of Banach function spaces on $\R$
satisfying the conditions \ref{it:ou1} and \ref{it:ou2}
are given.


\end{document}